\documentclass[11pt,letterpaper]{article}
\usepackage[margin=1.5in]{geometry}
\usepackage{amsmath,times}
\usepackage{amssymb,amsmath}
\usepackage{graphicx}
\usepackage{amsthm}
\usepackage{multicol}
\usepackage[all]{xy}
\usepackage{cancel}
\usepackage{tikz}
\usepackage[utf8]{inputenc}
\usepackage{url}
\usepackage{authblk}
\usepackage{booktabs}
\usepackage{xcolor}
\usepackage[dvipsnames]{xcolor}
\usepackage{tcolorbox}
\tcbuselibrary{theorems}
\usetikzlibrary{tikzmark,positioning,cd,arrows.meta}
\usepackage{graphicx}
\usepackage{multicol}
\usepackage{tikz}
\usepackage{doi}
\usepackage{tikz-cd}
\usepackage{orcidlink}
\usepackage{url}
\usepackage{authblk}
\usepackage{booktabs}
\usepackage{xcolor}
\usepackage[dvipsnames]{xcolor}
\usepackage{tcolorbox}
\tcbuselibrary{theorems}
\usetikzlibrary{tikzmark,positioning,cd,arrows.meta}

\title{A note on simple actions of small monoids}

\author[1]{Ram\'on H. Ruiz-Medina \orcidlink{0000-0003-2916-9160}}

\affil[1]{Centro de Ciencias Matem\'aticas.
Universidad Nacional Aut\'onoma de M\'exico.
Morelia, Michoac\'an, M\'exico. CP 58089 \\ \texttt{harath.ruiz@academicos.udg.mx}}

\date{}

\newtheorem{teorema}{Theorem}[]
\newtheorem{lema}[teorema]{Lemma}

\newcommand{\zz}{\mathbb{Z}}

\begin{document}

\maketitle

\begin{abstract}
These notes aim to exemplify the concept of action and partial action of a finite monoid $M$, identifying among the so-called \emph{indecomposable} actions those that are also \emph{simple}. We present all the simplified examples of monoids (non-groups) of order two and three, explicitly visualizing which are simple and which are not.

In addition, we prove three general lemmas relating simplicity to indecomposability in the context of partial actions, and establish a result on the relation between sub-actions of total actions and partial actions. Finally, we point out the connection of this work with the Burnside ring of monoids introduced by Weissmann and with Webb's theory of $C$-sets.


\end{abstract}

\section{Introduction}

The study of group actions on sets is one of the pillars of modern algebra. Since Burnside's foundational work \cite{burnside1909theory}, the theory of group actions has been fundamental in representation theory, combinatorics, and topology. The Burnside ring of a finite group, which encodes the actions of the group on finite sets, was systematically developed by Dress \cite{dress1969characterisation} and Bouc \cite{Bouc}, and has been the subject of intense study \cite{spec-SCHW, raggi2005groups, irreducible}.

The observation that a group can be viewed as a category with a single object opened the door to generalizations. Webb \cite{Webb2023} introduced $C$-sets as functors from a finite category to finite sets, defining the Burnside ring of the category. In that work, Webb shows that the Burnside ring of a finite category does not always have finite rank. Recently, Calder\'on, Raggi-C\'ardenas, Rosas and Ruiz-Medina \cite{semisimple} characterized the categories whose Burnside ring has finite rank: they are exactly the groupoids. On the other hand, Weissmann \cite{weiss} extended the construction to monoids, defining the Burnside ring of a monoid and distinguishing between weak and strong orbits.

The present work lies in this line. Our goal is to classify, explicitly and computationally, the actions (total and partial) of the non-group monoids of order 2 and 3, identifying those that are simple. We adopt the categorical point of view: an action of a monoid $M$ is a functor from its delooping category to the category of sets (for total actions) or to the category of sets with partial functions (for partial actions). The classification was obtained by an exhaustive computer search. For each possible table, the fulfillment of the action axioms was automatically verified, then the isomorphism classes were identified and finally it was determined which were simple and indecomposable.

In Section 2 we establish the basic definitions. In Section 3 we classify the unique monoid of order 2. In Section 4 we classify the six monoids of order 3, presenting tables and diagrams for each action. Finally, in Section 5 we prove three general lemmas relating simplicity, indecomposability and sub-actions.

\section{Preliminaries}

Throughout this work, $M$ will denote a finite monoid with identity element $1$. An \emph{$M$-set} (total) is a set $X$ equipped with an operation $\cdot: M\times X \to X$ such that $1\cdot x = x$ and $(mn)\cdot x = m\cdot (n\cdot x)$.

A \emph{partial action} of $M$ on $X$ is a partial operation $\cdot: M\times X \rightharpoonup X$ satisfying $1\cdot x = x$ for all $x$, and the compatibility condition: $m\cdot (n\cdot x)$ is defined if and only if $(mn)\cdot x$ is defined, and in that case they coincide. Equivalently, it is a monoid morphism $M \to \operatorname{PT}(X)$.

An $M$-subset is an invariant subset under the action. An $M$-set is \emph{indecomposable} if it is not a disjoint union of two non-empty proper $M$-subsets. It is \emph{simple} if it has no non-empty proper $M$-subsets. Every simple set is indecomposable, but the converse does not always hold in the total case.

A morphism of $M$-sets is a function $f:X\to Y$ that commutes with the action: $f(m\cdot x)=m\cdot f(x)$, understanding equality in the partial sense. If it is bijective, it is an isomorphism.

Actions will be represented by \emph{action tables}, where the entry in row $m$ and column $x$ is $m\cdot x$, or $\emptyset$ if it is not defined. Two tables are isomorphic if there is a bijection between the underlying sets that commutes with the actions.

\section{The unique monoid of order 2}

Given a two-element set on which an associative binary operation with identity is defined, but which does not define a group, there is only one case. For this case we define and work with the monoid from its Cayley table.

$$\begin{array}{c|cc}
M& 1&a\\
\hline 
1&1&a\\
a&a&a
\end{array}$$

\subsection{Sets with 2 elements}

For a two-element set $X=\{x,y\}$, we define partial actions on it through what we will call \emph{action tables}. We also generate a small diagram to visualize them and determine whether they are total or partial, indecomposable and/or simple. We will only place one representative per isomorphism class of $M$-sets.\\

From their action tables, of which there are 9 possibilities (since the first row must be invariant), it was identified that 6 define an action, either total or partial. Among these, 4 isomorphism classes were located, which are presented below in the following table.

\begin{center}
\begin{tabular}{|c|ccc|c|}
\hline 
$\begin{array}{|c|cc|}\hline & x&y\\\hline 1&x&y\\a&x&y\\ \hline \end{array}$
&& $\xymatrix{x\ar_{1,a}@(ul,dl)&y\ar^{1,a}@(ur,dr)}$ && \begin{tabular}{c}Total action \\ Indecomposable: NO\\ Simple: NO  \end{tabular}\\
\hline
$\begin{array}{|c|cc|}\hline & x&y\\\hline 1&x&y\\a&y&y\\ \hline \end{array}$
&& $\xymatrix{x\ar_{1}@(ul,dl)\ar^{a}[r]&y\ar^{1,a}@(ur,dr)}$ && \begin{tabular}{c}Total action \\ Indecomposable: YES\\ Simple: NO  \end{tabular}\\
\hline
$\begin{array}{|c|cc|}\hline & x&y\\\hline 1&x&y\\a&\emptyset&\emptyset\\ \hline \end{array}$
&& $\xymatrix{x\ar_{1}@(ul,dl)&y\ar^{1}@(ur,dr)}$
&& \begin{tabular}{c}Partial action \\ Indecomposable: NO\\ Simple: NO  \end{tabular}\\
\hline
$\begin{array}{|c|cc|}\hline & x&y\\\hline 1&x&y\\a&x&\emptyset\\ \hline \end{array}$
&& $\xymatrix{x\ar_{1,a}@(ul,dl)&y\ar^{1}@(ur,dr)}$
&& \begin{tabular}{c}Partial action \\ Indecomposable: YES\\ Simple: NO  \end{tabular}\\
\hline
\end{tabular}
\end{center}

\subsection{Sets with 3 elements}
No three-element set with a partial or total action defined from the two-element monoid defined in this section is simple. Bearing in mind that identifying simple elements is the main objective of this note, this section remains partially empty, with the comment that an important result explaining this phenomenon is introduced in the last section.

\section{Monoids of order 3}

In this section we present analogous tables for the 6 monoids of order 3 that exist, which are not groups. We divide by subsections, speaking about one monoid in each.

\subsection{The monoid $\zz_{2}$ with zero}

We define the operation of our monoid through the following Cayley table.

$$\begin{array}{c|ccc}&1&-1&0 \\ \hline 1&1&-1&0\\ -1&-1&1&0\\ 0&0&0&0 \end{array}$$

For this monoid the following is identified.

For one-element sets: 2 total or partial actions, which are not isomorphic, i.e. there are 2 classes, and both are simple.

\begin{center}
\begin{tabular}{|c|ccc|c|}
\hline 
$\begin{array}{|c|c|}\hline
& x\\\hline
1&x\\-1&x\\ 0&x\\ \hline \end{array}$
&& $\xymatrix{x\ar^{1,-1,0}@(ur,dr)}$ && \begin{tabular}{c}Total action \\ Indecomposable: YES\\ Simple: YES  \end{tabular}\\
\hline
$\begin{array}{|c|c|}\hline
& x\\\hline
1&x\\-1&x\\ 0&\emptyset\\ \hline \end{array}$
&& $\xymatrix{x\ar^{1,-1}@(ur,dr)}$ && \begin{tabular}{c}Partial action \\ Indecomposable: YES\\ Simple: YES  \end{tabular}\\
\hline
\end{tabular}
\end{center}

For two-element sets: out of the 81 possible tables, 7 total or partial actions, of which 5 isomorphism classes are identified, and one is simple.

\begin{center}
\begin{tabular}{|c|ccc|c|}
\hline 
$\begin{array}{|c|cc|}\hline
& x&y\\\hline
1&x&y\\-1&x&y\\ 0&x&y\\ \hline \end{array}$
&& $\xymatrix{x\ar_{1,-1,0}@(ul,dl)&y\ar^{1,-1,0}@(ur,dr)}$ && \begin{tabular}{c}Total action \\ Indecomposable: NO\\ Simple: NO  \end{tabular}\\
\hline
$\begin{array}{|c|cc|}\hline
& x&y\\\hline
1&x&y\\-1&x&y\\ 0&y&y\\ \hline \end{array}$
&& $\xymatrix{x\ar_{1,-1}@(ul,dl)\ar^{0}[r]&y\ar^{1,-1,0}@(ur,dr)}$ && \begin{tabular}{c}Total action \\ Indecomposable: NO\\ Simple: NO  \end{tabular}\\
\hline

$\begin{array}{|c|cc|}\hline
& x&y\\\hline
1&x&y\\-1&x&y\\ 0&x&\emptyset\\ \hline \end{array}$
&& $\xymatrix{x\ar_{1,-1,0}@(ul,dl)&y\ar^{1,-1}@(ur,dr)}$ && \begin{tabular}{c}Partial action \\ Indecomposable: NO\\ Simple: NO  \end{tabular}\\
\hline

$\begin{array}{|c|cc|}\hline
& x&y\\\hline
1&x&y\\-1&x&y\\ 0&\emptyset&\emptyset\\ \hline \end{array}$
&& $\xymatrix{x\ar_{1,-1}@(ul,dl)&y\ar^{1,-1}@(ur,dr)}$ && \begin{tabular}{c}Partial action \\ Indecomposable: NO\\ Simple: NO  \end{tabular}\\
\hline

$\begin{array}{|c|cc|}\hline
& x&y\\\hline
1&x&y\\-1&y&x\\ 0&\emptyset&\emptyset\\ \hline \end{array}$
&& $\xymatrix{x\ar_{1}@(ul,dl)\ar@/^/^{-1}[r]&y\ar@/^/^{-1}[l]\ar^{1}@(ur,dr)}$ && \begin{tabular}{c}Partial action \\ Indecomposable: NO\\ Simple: YES  \end{tabular}\\
\hline
\end{tabular}
\end{center}

For actions on three-element sets: out of the 4096 possible tables, 35 total or partial actions are identified, of which 11 isomorphism classes are identified, none of which are simple.

\begin{center}
\begin{tabular}{|c|ccc|c|}
\hline 
$\begin{array}{|c|ccc|}\hline
& x&y&z\\\hline
1&x&y&z\\-1&x&y&z\\ 0&x&y&z\\ \hline \end{array}$
&& $\xymatrix{x\ar_{1,-1,0}@(ul,dl)& & y\ar^{1,-1,0}@(ur,dr) \\ & z\ar^{1,-1,0}@(ul,ur)}$ && \begin{tabular}{c}Total action \\ Indecomposable: NO\\ Simple: NO  \end{tabular}\\
\hline 
$\begin{array}{|c|ccc|}\hline
& x&y&z\\\hline
1&x&y&z\\-1&x&y&z\\ 0&\emptyset&\emptyset&\emptyset\\ \hline \end{array}$
&& $\xymatrix{x\ar_{1,-1}@(ul,dl)& & y\ar^{1,-1}@(ur,dr) \\ & z\ar^{1,-1}@(ul,ur)}$ && \begin{tabular}{c}Partial action \\ Indecomposable: NO\\ Simple: NO  \end{tabular}\\
\hline 
$\begin{array}{|c|ccc|}\hline
& x&y&z\\\hline
1&x&y&z\\-1&x&y&z\\ 0&x&\emptyset&\emptyset\\ \hline \end{array}$
&& $\xymatrix{x\ar_{1,-1,0}@(ul,dl)& & y\ar^{1,-1}@(ur,dr) \\ & z\ar^{1,-1}@(ul,ur)}$ && \begin{tabular}{c}Partial action \\ Indecomposable: NO\\ Simple: NO  \end{tabular}\\
\hline 
$\begin{array}{|c|ccc|}\hline
& x&y&z\\\hline
1&x&y&z\\-1&x&y&z\\ 0&y&y&\emptyset\\ \hline \end{array}$
&& $\xymatrix{x\ar_{1,-1}@(ul,dl) \ar^{0}[rr] & & y\ar^{1,-1,0}@(ur,dr) \\ & z\ar_{1,-1}@(dl,dr)}$ && \begin{tabular}{c}Partial action \\ Indecomposable: NO\\ Simple: NO  \end{tabular}\\
\hline 
$\begin{array}{|c|ccc|}\hline
& x&y&z\\\hline
1&x&y&z\\-1&x&y&z\\ 0&x&y&\emptyset\\ \hline \end{array}$
&& $\xymatrix{x\ar_{1,-1,0}@(ul,dl)& & y\ar^{1,-1,0}@(ur,dr) \\ & z\ar^{1,-1}@(ul,ur)}$ && \begin{tabular}{c}Partial action \\ Indecomposable: NO\\ Simple: NO  \end{tabular}\\
\hline 
$\begin{array}{|c|ccc|}\hline
& x&y&z\\\hline
1&x&y&z\\-1&x&y&z\\ 0&z&z&z\\ \hline \end{array}$
&& $\xymatrix{x\ar_{1,-1}@(ul,dl)\ar_{0}[dr]& & y\ar^{1,-1}@(ur,dr) \ar^{0}[dl] \\ & z\ar_{1,-1,0}@(dl,dr) }$ && \begin{tabular}{c}Total action \\ Indecomposable: NO\\ Simple: NO  \end{tabular}\\
\hline 
\end{tabular}
\end{center}

\begin{center}
\begin{tabular}{|c|ccc|c|}
\hline 
$\begin{array}{|c|ccc|}\hline
& x&y&z\\\hline
1&x&y&z\\-1&x&y&z\\ 0&x&y&x\\ \hline \end{array}$
&& $\xymatrix{x\ar_{1,-1,0}@(ul,dl)& & y\ar^{1,-1,0}@(ur,dr) \\ & z\ar^{1,-1}@(ul,ur) \ar^{0}[ul]}$ && \begin{tabular}{c}Total action \\ Indecomposable: NO\\ Simple: NO  \end{tabular}\\
\hline 

$\begin{array}{|c|ccc|}\hline
& x&y&z\\\hline
1&x&y&z\\-1&y&x&z\\ 0&\emptyset&\emptyset&\emptyset\\ \hline \end{array}$
&& $\xymatrix{x\ar_{1}@(ul,dl) \ar@/^/^{-1}[rr] & & y\ar^{1}@(ur,dr) \ar@/^/^{-1}[ll] \\ & z\ar_{1,-1}@(dl,dr) &}$ && \begin{tabular}{c}Partial action \\ Indecomposable: NO\\ Simple: NO  \end{tabular}\\
\hline 
$\begin{array}{|c|ccc|}\hline
& x&y&z\\\hline
1&x&y&z\\-1&y&x&z\\ 0&\emptyset&\emptyset&z\\ \hline \end{array}$
&& $\xymatrix{x\ar_{1}@(ul,dl) \ar@/^/^{-1}[rr] & & y\ar^{1}@(ur,dr) \ar@/^/^{-1}[ll] \\ & z\ar_{1,-1,0}@(dl,dr) &}$ && \begin{tabular}{c}Partial action \\ Indecomposable: NO\\ Simple: NO  \end{tabular}\\
\hline 
$\begin{array}{|c|ccc|}\hline
& x&y&z\\\hline
1&x&y&z\\-1&y&x&z\\ 0&z&z&\emptyset\\ \hline \end{array}$
&& $\xymatrix{x \ar_{1}@(ul,dl) \ar@/^/^{-1}[rr] \ar_{0}[dr] & & 
  y \ar^{1}@(ur,dr) \ar@/^/^{-1}[ll] \ar^{0}[ld] \\
  & z \ar_{1,-1}@(dl,dr) &}$ && \begin{tabular}{c}Partial action \\ Indecomposable: NO\\ Simple: NO  \end{tabular}\\
\hline 
$\begin{array}{|c|ccc|}\hline
& x&y&z\\\hline
1&x&y&z\\-1&y&x&z\\ 0&z&z&z\\ \hline \end{array}$
&& $\xymatrix{x\ar_{1}@(ul,dl)\ar@/^/^{-1}[rr]\ar_{0}[dr]& & y\ar^{1}@(ur,dr)\ar@/^/^{-1}[ll]\ar^{0}[ld] \\ & z\ar_{1,-1,0}@(dl,dr)}$ && \begin{tabular}{c}Total action \\ Indecomposable: NO\\ Simple: NO  \end{tabular}\\
\hline

\end{tabular}
\end{center}

\subsection{The monoid with an idempotent and zero}

We define the operation of our monoid through the following Cayley table.

$$\begin{array}{c|ccc}&1&a&0 \\ \hline 1&1&a&0\\ a&a&a&0\\ 0&0&0&0 \end{array}$$

For this monoid the following is identified.

For one-element sets: 3 total or partial actions, which are not isomorphic, i.e. there are 3 classes, and all are simple.

\begin{center}
\begin{tabular}{|c|ccc|c|}
\hline 
$\begin{array}{|c|c|}\hline
& x\\\hline
1&x\\a&x\\ 0&x\\ \hline \end{array}$
&& $\xymatrix{x\ar^{1,a,0}@(ur,dr)}$ && \begin{tabular}{c}Total action \\ Indecomposable: YES\\ Simple: YES  \end{tabular}\\
\hline
$\begin{array}{|c|c|}\hline
& x\\\hline
1&x\\a&x\\ 0&\emptyset\\ \hline \end{array}$
&& $\xymatrix{x\ar^{1,a}@(ur,dr)}$ && \begin{tabular}{c}Partial action \\ Indecomposable: YES\\ Simple: YES  \end{tabular}\\
\hline
$\begin{array}{|c|c|}\hline
& x\\\hline
1&x\\a&\emptyset\\ 0&\emptyset\\ \hline \end{array}$
&& $\xymatrix{x\ar^{1}@(ur,dr)}$ && \begin{tabular}{c}Partial action \\ Indecomposable: YES\\ Simple: YES  \end{tabular}\\
\hline
\end{tabular}
\end{center}

For two-element sets: out of the 81 possible tables, 23 total or partial actions, of which 13 isomorphism classes are identified, and none are simple.

\begin{center}
\begin{tabular}{|c|ccc|c|}
\hline 
$\begin{array}{|c|cc|}\hline
& x&y\\\hline
1&x&y\\a&x&y\\ 0&x&y\\ \hline \end{array}$
&& $\xymatrix{x\ar_{1,a,0}@(ul,dl)&y\ar^{1,a,0}@(ur,dr)}$ && \begin{tabular}{c}Total action \\ Indecomposable: NO\\ Simple: NO  \end{tabular}\\
\hline
$\begin{array}{|c|cc|}\hline
& x&y\\\hline
1&x&y\\a&x&x\\ 0&x&x\\ \hline \end{array}$
&& $\xymatrix{x\ar_{1,a,0}@(ul,dl)&y\ar^{1}@(ur,dr)\ar_{a,0}[l]}$ && \begin{tabular}{c}Total action \\ Indecomposable: NO\\ Simple: NO  \end{tabular}\\
\hline
$\begin{array}{|c|cc|}\hline
& x&y\\\hline
1&x&y\\a&x&y\\ 0&x&x\\ \hline \end{array}$
&& $\xymatrix{x\ar_{1,a,0}@(ul,dl)&y\ar^{1,a}@(ur,dr)\ar_{0}[l]}$ && \begin{tabular}{c}Total action \\ Indecomposable: NO\\ Simple: NO  \end{tabular}\\
\hline
$\begin{array}{|c|cc|}\hline
& x&y\\\hline
1&x&y\\a&x&x\\ 0&x&y\\ \hline \end{array}$
&& $\xymatrix{x\ar_{1,a,0}@(ul,dl)&y\ar^{1,0}@(ur,dr)\ar_{a}[l]}$ && \begin{tabular}{c}Total action \\ Indecomposable: NO\\ Simple: NO  \end{tabular}\\
\hline
$\begin{array}{|c|cc|}\hline
& x&y\\\hline
1&x&y\\a&x&y\\ 0&x&\emptyset\\ \hline \end{array}$
&& $\xymatrix{x\ar_{1,a,0}@(ul,dl)&y\ar^{1,a}@(ur,dr)}$ && \begin{tabular}{c}Partial action \\ Indecomposable: NO\\ Simple: NO  \end{tabular}\\
\hline
\end{tabular}
\end{center}

\begin{center}
\begin{tabular}{|c|ccc|c|}
\hline 

$\begin{array}{|c|cc|}\hline
& x&y\\\hline
1&x&y\\a&x&x\\ 0&x&\emptyset\\ \hline \end{array}$
&& $\xymatrix{x\ar_{1,a,0}@(ul,dl)&y\ar^{1}@(ur,dr)\ar_{a}[l]}$ && \begin{tabular}{c}Partial action \\ Indecomposable: NO\\ Simple: NO  \end{tabular}\\
\hline
$\begin{array}{|c|cc|}\hline
& x&y\\\hline
1&x&y\\a&x&x\\ 0&x&x\\ \hline \end{array}$
&& $\xymatrix{x\ar_{1,a,0}@(ul,dl)&y\ar^{1}@(ur,dr)\ar_{a,0}[l]}$ && \begin{tabular}{c}Partial action \\ Indecomposable: NO\\ Simple: NO  \end{tabular}\\
\hline
$\begin{array}{|c|cc|}\hline
& x&y\\\hline
1&x&y\\a&x&x\\ 0&x&\emptyset\\ \hline \end{array}$
&& $\xymatrix{x\ar_{1,a,0}@(ul,dl)&y\ar^{1}@(ur,dr)\ar_{a}[l]}$ && \begin{tabular}{c}Partial action \\ Indecomposable: NO\\ Simple: NO  \end{tabular}\\
\hline
$\begin{array}{|c|cc|}\hline
& x&y\\\hline
1&x&y\\a&x&\emptyset\\ 0&x&y\\ \hline \end{array}$
&& $\xymatrix{x\ar_{1,a,0}@(ul,dl)&y\ar^{1,0}@(ur,dr)}$ && \begin{tabular}{c}Partial action \\ Indecomposable: NO\\ Simple: NO  \end{tabular}\\
\hline
$\begin{array}{|c|cc|}\hline
& x&y\\\hline
1&x&y\\a&x&\emptyset\\ 0&x&x\\ \hline \end{array}$
&& $\xymatrix{x\ar_{1,a,0}@(ul,dl)&y\ar^{1}@(ur,dr)\ar_{0}[l]}$ && \begin{tabular}{c}Partial action \\ Indecomposable: NO\\ Simple: NO  \end{tabular}\\
\hline
$\begin{array}{|c|cc|}\hline
& x&y\\\hline
1&x&y\\a&x&\emptyset\\ 0&x&\emptyset\\ \hline \end{array}$
&& $\xymatrix{x\ar_{1,a,0}@(ul,dl)&y\ar^{1}@(ur,dr)}$ && \begin{tabular}{c}Partial action \\ Indecomposable: NO\\ Simple: NO  \end{tabular}\\
\hline
$\begin{array}{|c|cc|}\hline
& x&y\\\hline
1&x&y\\a&x&\emptyset\\ 0&\emptyset&\emptyset\\ \hline \end{array}$
&& $\xymatrix{x\ar_{1,a}@(ul,dl)&y\ar^{1}@(ur,dr)}$ && \begin{tabular}{c}Partial action \\ Indecomposable: NO\\ Simple: NO  \end{tabular}\\
\hline
$\begin{array}{|c|cc|}\hline
& x&y\\\hline
1&x&y\\a&\emptyset&\emptyset\\ 0&\emptyset&\emptyset\\ \hline \end{array}$
&& $\xymatrix{x\ar_{1}@(ul,dl)&y\ar^{1}@(ur,dr)}$ && \begin{tabular}{c}Partial action \\ Indecomposable: NO\\ Simple: NO  \end{tabular}\\
\hline
\end{tabular}
\end{center}

For actions on three-element sets: out of the 4096 possible tables, 57 total or partial actions are identified, of which 14 isomorphism classes are identified, none of which are simple.

\begin{center}
\begin{tabular}{|c|ccc|c|}
\hline 
$\begin{array}{|c|ccc|}\hline
& x&y&z\\\hline
1&x&y&z\\a&x&y&z\\ 0&x&y&z\\ \hline \end{array}$
&& $\xymatrix{x\ar_{1,a,0}@(ul,dl)& & y\ar^{1,a,0}@(ur,dr) \\ & z\ar^{1,a,0}@(ul,ur)}$ && \begin{tabular}{c}Total action \\ Indecomposable: NO\\ Simple: NO  \end{tabular}\\
\hline\end{tabular}
\end{center}

\begin{center}
\begin{tabular}{|c|ccc|c|}
\hline 

$\begin{array}{|c|ccc|}\hline
& x&y&z\\\hline
1&x&y&z\\a&x&x&x\\ 0&x&x&x\\ \hline \end{array}$
&& $\xymatrix{x\ar_{1,a,0}@(ul,dl)& & y\ar^{1}@(ur,dr)\ar_{a,0}[ll] \\ & z\ar^{1}@(ul,ur)\ar^{a,0}[lu]}$ && \begin{tabular}{c}Total action \\ Indecomposable: NO\\ Simple: NO  \end{tabular}\\
\hline
$\begin{array}{|c|ccc|}\hline
& x&y&z\\\hline
1&x&y&z\\a&x&x&x\\ 0&x&y&z\\ \hline \end{array}$
&& $\xymatrix{x\ar_{1,a,0}@(ul,dl)& & y\ar^{1,0}@(ur,dr)\ar_{a}[ll] \\ & z\ar^{1,0}@(ul,ur)\ar^{a}[lu]}$ && \begin{tabular}{c}Total action \\ Indecomposable: NO\\ Simple: NO  \end{tabular}\\
\hline
$\begin{array}{|c|ccc|}\hline
& x&y&z\\\hline
1&x&y&z\\a&x&y&z\\ 0&x&x&x\\ \hline \end{array}$
&& $\xymatrix{x\ar_{1,a,0}@(ul,dl)& & y\ar^{1,a}@(ur,dr)\ar_{0}[ll] \\ & z\ar^{1,a}@(ul,ur)\ar^{0}[lu]}$ && \begin{tabular}{c}Total action \\ Indecomposable: NO\\ Simple: NO  \end{tabular}\\
\hline
$\begin{array}{|c|ccc|}\hline
& x&y&z\\\hline
1&x&y&z\\a&x&y&z\\ 0&x&y&x\\ \hline \end{array}$
&& $\xymatrix{x\ar_{1,a,0}@(ul,dl)& & y\ar^{1,a,0}@(ur,dr) \\ & z\ar^{1,a}@(ul,ur)\ar^{0}[ul]}$ && \begin{tabular}{c}Total action \\ Indecomposable: NO\\ Simple: NO  \end{tabular}\\
\hline
$\begin{array}{|c|ccc|}\hline
& x&y&z\\\hline
1&x&y&z\\a&x&y&x\\ 0&x&y&x\\ \hline \end{array}$
&& $\xymatrix{x\ar_{1,a,0}@(ul,dl)& & y\ar^{1,a,0}@(ur,dr) \\ & z\ar^{1}@(ul,ur)\ar^{a,0}[ul]}$ && \begin{tabular}{c}Total action \\ Indecomposable: NO\\ Simple: NO  \end{tabular}\\
\hline
$\begin{array}{|c|ccc|}\hline
& x&y&z\\\hline
1&x&y&z\\a&x&y&z\\ 0&\emptyset&\emptyset&\emptyset\\ \hline \end{array}$
&& $\xymatrix{x\ar_{1,a}@(ul,dl)& & y\ar^{1,a}@(ur,dr) \\ & z\ar^{1,a}@(ul,ur)}$ && \begin{tabular}{c}Partial action \\ Indecomposable: NO\\ Simple: NO  \end{tabular}\\
\hline
$\begin{array}{|c|ccc|}\hline
& x&y&z\\\hline
1&x&y&z\\a&x&y&z\\ 0&x&\emptyset&\emptyset\\ \hline \end{array}$
&& $\xymatrix{x\ar_{1,a,0}@(ul,dl)& & y\ar^{1,a}@(ur,dr) \\ & z\ar^{1,a}@(ul,ur)}$ && \begin{tabular}{c}Partial action \\ Indecomposable: NO\\ Simple: NO  \end{tabular}\\
\hline

$\begin{array}{|c|ccc|}\hline
& x&y&z\\\hline
1&x&y&z\\a&x&y&z\\ 0&x&x&\emptyset\\ \hline \end{array}$
&& $\xymatrix{x\ar_{1,a,0}@(ul,dl)& & y\ar^{1,a}@(ur,dr)\ar_{0}[ll] \\ & z\ar^{1,a}@(ul,ur)}$ && \begin{tabular}{c}Partial action \\ Indecomposable: NO\\ Simple: NO  \end{tabular}\\
\hline
\end{tabular}
\end{center}

\begin{center}
\begin{tabular}{|c|ccc|c|}
\hline 

$\begin{array}{|c|ccc|}\hline
& x&y&z\\\hline
1&x&y&z\\a&x&y&z\\ 0&x&y&\emptyset\\ \hline \end{array}$
&& $\xymatrix{x\ar_{1,a,0}@(ul,dl)& & y\ar^{1,a,0}@(ur,dr) \\ & z\ar^{1,a}@(ul,ur)}$ && \begin{tabular}{c}Partial action \\ Indecomposable: NO\\ Simple: NO  \end{tabular}\\
\hline
$\begin{array}{|c|ccc|}\hline
& x&y&z\\\hline
1&x&y&z\\a&x&x&x\\ 0&x&y&\emptyset\\ \hline \end{array}$
&& $\xymatrix{x\ar_{1,a,0}@(ul,dl)& & y\ar^{1,0}@(ur,dr)\ar_{a}[ll] \\ & z\ar^{1}@(ul,ur)}$ && \begin{tabular}{c}Partial action \\ Indecomposable: NO\\ Simple: NO  \end{tabular}\\
\hline
$\begin{array}{|c|ccc|}\hline
& x&y&z\\\hline
1&x&y&z\\a&x&y&x\\ 0&x&y&\emptyset\\ \hline \end{array}$
&& $\xymatrix{x\ar_{1,a,0}@(ul,dl)& & y\ar^{1,a,0}@(ur,dr) \\ & z\ar^{1}@(ul,ur)\ar^{a}[ul]}$ && \begin{tabular}{c}Partial action \\ Indecomposable: NO\\ Simple: NO  \end{tabular}\\
\hline
$\begin{array}{|c|ccc|}\hline
& x&y&z\\\hline
1&x&y&z\\a&x&x&x\\ 0&x&x&\emptyset\\ \hline \end{array}$
&& $\xymatrix{x\ar_{1,a,0}@(ul,dl)& & y\ar^{1}@(ur,dr)\ar_{a,0}[ll] \\ & z\ar^{1}@(ul,ur)}$ && \begin{tabular}{c}Partial action \\ Indecomposable: NO\\ Simple: NO  \end{tabular}\\
\hline
$\begin{array}{|c|ccc|}\hline
& x&y&z\\\hline
1&x&y&z\\a&x&\emptyset&\emptyset\\ 0&x&\emptyset&\emptyset\\ \hline \end{array}$
&& $\xymatrix{x\ar_{1,a,0}@(ul,dl)& & y\ar^{1}@(ur,dr) \\ & z\ar^{1}@(ul,ur)}$ && \begin{tabular}{c}Partial action \\ Indecomposable: NO\\ Simple: NO  \end{tabular}\\
\hline
\end{tabular}
\end{center}

\subsection{Monoid $M_3$: $a^2=a,\ b^2=b,\ ab=a,\ ba=b$}

We define the operation of our monoid through the following Cayley table.

$$\begin{array}{c|ccc}&1&a&b \\ \hline 1&1&a&b\\ a&a&a&a\\ b&b&b&b \end{array}$$

For one-element sets: 2 actions, 2 classes, both simple.

\begin{center}
\begin{tabular}{|c|ccc|c|}
\hline 
$\begin{array}{|c|c|}\hline
& x\\\hline
1&x\\a&x\\ b&x\\ \hline \end{array}$
&& $\xymatrix{x\ar^{1,a,b}@(ur,dr)}$ && \begin{tabular}{c}Total action \\ Indecomposable: YES\\ Simple: YES  \end{tabular}\\
\hline\end{tabular}
\end{center}

\begin{center}
\begin{tabular}{|c|ccc|c|}
\hline 

$\begin{array}{|c|c|}\hline
& x\\\hline
1&x\\a&\emptyset\\ b&\emptyset\\ \hline \end{array}$
&& $\xymatrix{x\ar^{1}@(ur,dr)}$ && \begin{tabular}{c}Partial action \\ Indecomposable: YES\\ Simple: YES  \end{tabular}\\
\hline
\end{tabular}
\end{center}

For two-element sets: out of 81 tables, 5 actions, 3 classes, 1 simple.

\begin{center}
\begin{tabular}{|c|ccc|c|}
\hline 
$\begin{array}{|c|cc|}\hline
& x&y\\\hline
1&x&y\\a&x&y\\ b&x&y\\ \hline \end{array}$
&& $\xymatrix{x\ar_{1,a,b}@(ul,dl)&y\ar^{1,a,b}@(ur,dr)}$ && \begin{tabular}{c}Total action \\ Indecomposable: NO\\ Simple: NO  \end{tabular}\\
\hline
$\begin{array}{|c|cc|}\hline
& x&y\\\hline
1&x&y\\a&x&x\\ b&x&x\\ \hline \end{array}$
&& $\xymatrix{x\ar_{1,a,b}@(ul,dl)&y\ar^{1}@(ur,dr)\ar_{a,b}[l]}$ && \begin{tabular}{c}Total action \\ Indecomposable: NO\\ Simple: NO  \end{tabular}\\
\hline
$\begin{array}{|c|cc|}\hline
& x&y\\\hline
1&x&y\\a&y&y\\ b&x&x\\ \hline \end{array}$
&& $\xymatrix{x\ar_{1,b}@(ul,dl)\ar@/^/^{a}[r]&y\ar^{1,a}@(ur,dr)\ar^{b}@/^/[l]}$ && \begin{tabular}{c}Total action \\ Indecomposable: YES\\ Simple: YES  \end{tabular}\\
\hline
$\begin{array}{|c|cc|}\hline
& x&y\\\hline
1&x&y\\a&\emptyset&\emptyset\\ b&\emptyset&\emptyset\\ \hline \end{array}$
&& $\xymatrix{x\ar_{1}@(ul,dl)&y\ar^{1}@(ur,dr)}$ && \begin{tabular}{c}Partial action \\ Indecomposable: NO\\ Simple: NO  \end{tabular}\\
\hline
$\begin{array}{|c|cc|}\hline
& x&y\\\hline
1&x&y\\a&x&\emptyset\\ b&x&\emptyset\\ \hline \end{array}$
&& $\xymatrix{x\ar_{1,a,b}@(ul,dl)&y\ar^{1}@(ur,dr)}$ && \begin{tabular}{c}Partial action \\ Indecomposable: NO\\ Simple: NO  \end{tabular}\\
\hline
\end{tabular}
\end{center}

For three-element sets: out of the 4096 possible tables, 41 actions are identified (22 total and 19 non-total partial), which are grouped into 10 isomorphism classes (5 total and 5 partial), and none of them is simple.
\begin{center}
\begin{tabular}{|c|ccc|c|}
\hline 
$\begin{array}{|c|ccc|}\hline
& x&y&z\\\hline
1&x&y&z\\a&x&y&z\\ b&x&y&z\\ \hline \end{array}$
&& $\xymatrix{x\ar_{1,a,b}@(ul,dl)& & y\ar^{1,a,b}@(ur,dr) \\ & z\ar^{1,a,b}@(ul,ur)}$ && \begin{tabular}{c}Total action \\ Indecomposable: NO\\ Simple: NO  \end{tabular}\\
\hline
$\begin{array}{|c|ccc|}\hline
& x&y&z\\\hline
1&x&y&z\\a&x&x&x\\ b&x&x&x\\ \hline \end{array}$
&& $\xymatrix{x\ar_{1,a,b}@(ul,dl)& & y\ar^{1}@(ur,dr)\ar_{a,b}[ll] \\ & z\ar^{1}@(ul,ur)\ar^{a,b}[lu]}$ && \begin{tabular}{c}Total action \\ Indecomposable: NO\\ Simple: NO  \end{tabular}\\
\hline
\end{tabular}
\end{center}

\begin{center}
\begin{tabular}{|c|ccc|c|}
\hline 
$\begin{array}{|c|ccc|}\hline
& x&y&z\\\hline
1&x&y&z\\a&x&x&x\\ b&y&y&y\\ \hline \end{array}$
&& $\xymatrix{x\ar_{1,a}@(ul,dl)\ar@/^/^{b}[rr]& & y\ar^{1,b}@(ur,dr)\ar@/^/^{a}[ll] \\ & z\ar_{1}@(dl,dr) \ar^{a}[lu] \ar_{b}[ru]}$ && \begin{tabular}{c}Total action \\ Indecomposable: NO\\ Simple: NO  \end{tabular}\\
\hline
$\begin{array}{|c|ccc|}\hline
& x&y&z\\\hline
1&x&y&z\\a&x&y&x\\ b&x&y&x\\ \hline \end{array}$
&& $\xymatrix{x\ar_{1,a,b}@(ul,dl)& & y\ar^{1,a,b}@(ur,dr) \\ & z\ar^{1}@(ul,ur)\ar^{a,b}[ul]}$ && \begin{tabular}{c}Total action \\ Indecomposable: NO\\ Simple: NO  \end{tabular}\\
\hline
$\begin{array}{|c|ccc|}\hline
& x&y&z\\\hline
1&x&y&z\\a&y&x&z\\ b&x&y&z\\ \hline \end{array}$
&& $\xymatrix{x\ar_{1,a}@(ul,dl)\ar^{b}@/^/[rr]& & y\ar^{1,b}@(ur,dr)\ar@/^/^{a}[ll] \\ & z\ar_{1,a,b}@(dl,dr)}$ && \begin{tabular}{c}Total action \\ Indecomposable: NO\\ Simple: NO  \end{tabular}\\
\hline 
$\begin{array}{|c|ccc|}\hline
& x&y&z\\\hline
1&x&y&z\\a&\emptyset&\emptyset&\emptyset\\ b&\emptyset&\emptyset&\emptyset\\ \hline \end{array}$
&& \xymatrix{x\ar_{1}@(ul,dl)& & y\ar^{1}@(ur,dr) \\ & z\ar^{1}@(ul,ur)} && \begin{tabular}{c}Partial action \\ Indecomposable: NO\\ Simple: NO  \end{tabular}\\
\hline 
$\begin{array}{|c|ccc|}\hline
& x&y&z\\\hline
1&x&y&z\\a&x&\emptyset&\emptyset\\ b&x&\emptyset&\emptyset\\ \hline \end{array}$
&& \xymatrix{x\ar_{1,a,b}@(ul,dl)& & y\ar^{1}@(ur,dr) \\ & z\ar^{1}@(ul,ur)} && \begin{tabular}{c}Partial action \\ Indecomposable: NO\\ Simple: NO  \end{tabular}\\
\hline 
$\begin{array}{|c|ccc|}\hline
& x&y&z\\\hline
1&x&y&z\\a&x&y&\emptyset\\ b&x&y&\emptyset\\ \hline \end{array}$
&& \xymatrix{x\ar_{1,a,b}@(ul,dl)& & y\ar^{1,a,b}@(ur,dr) \\ & z\ar^{1}@(ul,ur)} && \begin{tabular}{c}Partial action \\ Indecomposable: NO\\ Simple: NO  \end{tabular}\\
\hline 
$\begin{array}{|c|ccc|}\hline
& x&y&z\\\hline
1&x&y&z\\a&x&x&\emptyset\\ b&x&x&\emptyset\\ \hline \end{array}$
&& \xymatrix{x\ar_{1,a,b}@(ul,dl)& & y\ar^{1}@(ur,dr)\ar_{a,b}[ll] \\ & z\ar^{1}@(ul,ur)} && \begin{tabular}{c}Partial action \\ Indecomposable: NO\\ Simple: NO  \end{tabular}\\
\hline 
\end{tabular}
\end{center}

\begin{center}
\begin{tabular}{|c|ccc|c|}
\hline 
$\begin{array}{|c|ccc|}\hline
& x&y&z\\\hline
1&x&y&z\\a&y&y&\emptyset\\ b&x&x&\emptyset\\ \hline \end{array}$
&& \xymatrix{x\ar_{1,b}@(ul,dl)\ar^{a}[rr]& & y\ar^{1,a}@(ur,dr)\ar_{b}[ll] \\ & z\ar^{1}@(ul,ur)} && \begin{tabular}{c}Partial action \\ Indecomposable: NO\\ Simple: NO  \end{tabular}\\
\hline
\end{tabular}
\end{center}

\subsection{Monoid $M_4$: $a^2=a,\ b^2=b,\ ab=b,\ ba=a$}

Cayley table:
$$\begin{array}{c|ccc}&1&a&b \\ \hline 1&1&a&b\\ a&a&a&b\\ b&b&a&b \end{array}$$

For $|X|=1$: 2 classes, both simple.

\begin{center}
\begin{tabular}{|c|ccc|c|}
\hline 
$\begin{array}{|c|c|}\hline
& x\\\hline
1&x\\a&x\\ b&x\\ \hline \end{array}$
&& $\xymatrix{x\ar^{1,a,b}@(ur,dr)}$ && \begin{tabular}{c}Total action \\ Simple: YES  \end{tabular}\\
\hline
$\begin{array}{|c|c|}\hline
& x\\\hline
1&x\\a&\emptyset\\ b&\emptyset\\ \hline \end{array}$
&& $\xymatrix{x\ar^{1}@(ur,dr)}$ && \begin{tabular}{c}Partial action \\ Simple: YES  \end{tabular}\\
\hline
\end{tabular}
\end{center}

For $|X|=2$: out of 81 tables, 4 actions, 3 classes, none simple.

\begin{center}
\begin{tabular}{|c|ccc|c|}
\hline 
$\begin{array}{|c|cc|}\hline
& x&y\\\hline
1&x&y\\a&x&y\\ b&x&y\\ \hline \end{array}$
&& $\xymatrix{x\ar_{1,a,b}@(ul,dl)&y\ar^{1,a,b}@(ur,dr)}$ && \begin{tabular}{c}Total action \\ Not simple  \end{tabular}\\
\hline
$\begin{array}{|c|cc|}\hline
& x&y\\\hline
1&x&y\\a&x&x\\ b&x&x\\ \hline \end{array}$
&& $\xymatrix{x\ar_{1,a,b}@(ul,dl)&y\ar^{1}@(ur,dr)\ar_{a,b}[l]}$ && \begin{tabular}{c}Total action \\ Not simple  \end{tabular}\\
\hline
$\begin{array}{|c|cc|}\hline
& x&y\\\hline
1&x&y\\a&x&x\\ b&y&y\\ \hline \end{array}$
&& $\xymatrix{x\ar_{1,a}@(ul,dl)\ar@/^/^{b}[r]&y\ar^{1,b}@(ur,dr) \ar@/^/^{a}[l] }$ && \begin{tabular}{c}Total action \\ Not simple  \end{tabular}\\
\hline
$\begin{array}{|c|cc|}\hline
& x&y\\\hline
1&x&y\\a&\emptyset&\emptyset\\ b&\emptyset&\emptyset\\ \hline \end{array}$
&& $\xymatrix{x\ar_{1}@(ul,dl)&y\ar^{1}@(ur,dr)}$ && \begin{tabular}{c}Partial action \\ Not simple  \end{tabular}\\
\hline
\end{tabular}
\end{center}

For $|X|=3$: out of 4096 tables, 16 actions, 5 classes, none simple.

\begin{center}
\begin{tabular}{|c|ccc|c|}
\hline 
$\begin{array}{|c|ccc|}\hline
& x&y&z\\\hline
1&x&y&z\\a&x&y&z\\ b&x&y&z\\ \hline \end{array}$
&& $\xymatrix{x\ar_{1,a,b}@(ul,dl)& & y\ar^{1,a,b}@(ur,dr) \\ & z\ar^{1,a,b}@(ul,ur)}$ && \begin{tabular}{c}Total action \\ Not simple  \end{tabular}\\
\hline
$\begin{array}{|c|ccc|}\hline
& x&y&z\\\hline
1&x&y&z\\a&x&x&x\\ b&x&x&x\\ \hline \end{array}$
&& $\xymatrix{x\ar_{1,a,b}@(ul,dl)& & y\ar^{1}@(ur,dr)\ar_{a,b}[ll] \\ & z\ar_{1}@(dl,dr)\ar^{a,b}[lu]}$ && \begin{tabular}{c}Total action \\ Not simple  \end{tabular}\\
\hline
$\begin{array}{|c|ccc|}\hline
& x&y&z\\\hline
1&x&y&z\\a&x&x&z\\ b&x&y&z\\ \hline \end{array}$
&& $\xymatrix{x\ar_{1,a,b}@(ul,dl)& & y\ar^{1,b}@(ur,dr) \ar_{a}[ll] \\ & z\ar^{1,a,b}@(ul,ur)}$ && \begin{tabular}{c}Total action \\ Not simple  \end{tabular}\\
\hline
$\begin{array}{|c|ccc|}\hline
& x&y&z\\\hline
1&x&y&z\\a&x&x&x\\ b&y&y&x\\ \hline \end{array}$
&& $\xymatrix{x\ar_{1,a}@(ul,dl)& & y\ar^{1,b}@(ur,dr) \ar_{a}[ll] \\ & z\ar_{1}@(dl,dr) \ar^{a,b}[lu]}$ && \begin{tabular}{c}Total action \\ Not simple  \end{tabular}\\
\hline
$\begin{array}{|c|ccc|}\hline
& x&y&z\\\hline
1&x&y&z\\a&\emptyset&\emptyset&\emptyset\\ b&\emptyset&\emptyset&\emptyset\\ \hline \end{array}$
&& $\xymatrix{x\ar_{1}@(ul,dl)& & y\ar^{1}@(ur,dr) \\ & z\ar^{1}@(ul,ur)}$ && \begin{tabular}{c}Partial action \\ Not simple  \end{tabular}\\
\hline
\end{tabular}
\end{center}

\subsection{Monoid $M_5$: $b$ absorbent, $a^2=b,\ ab=ba=b$}

Cayley table:
$$\begin{array}{c|ccc}&1&a&b \\ \hline 1&1&a&b\\ a&a&b&b\\ b&b&b&b \end{array}$$

For $|X|=1$: 2 classes, both simple.

\begin{center}
\begin{tabular}{|c|ccc|c|}
\hline 
$\begin{array}{|c|c|}\hline
& x\\\hline
1&x\\a&x\\ b&x\\ \hline \end{array}$
&& $\xymatrix{x\ar^{1,a,b}@(ur,dr)}$ && \begin{tabular}{c}Total action \\ Simple: YES  \end{tabular}\\
\hline
$\begin{array}{|c|c|}\hline
& x\\\hline
1&x\\a&\emptyset\\ b&\emptyset\\ \hline \end{array}$
&& $\xymatrix{x\ar^{1}@(ur,dr)}$ && \begin{tabular}{c}Partial action \\ Simple: YES  \end{tabular}\\
\hline
\end{tabular}
\end{center}

For $|X|=2$: out of 81 tables, 3 actions, 2 classes, none simple.

\begin{center}
\begin{tabular}{|c|ccc|c|}
\hline 
$\begin{array}{|c|cc|}\hline
& x&y\\\hline
1&x&y\\a&x&y\\ b&x&y\\ \hline \end{array}$
&& $\xymatrix{x\ar_{1,a,b}@(ul,dl)&y\ar^{1,a,b}@(ur,dr)}$ && \begin{tabular}{c}Total action \\ Not simple  \end{tabular}\\
\hline
$\begin{array}{|c|cc|}\hline
& x&y\\\hline
1&x&y\\a&x&y\\ b&y&y\\ \hline \end{array}$
&& $\xymatrix{x\ar_{1,a}@(ul,dl)\ar^{b}[r]&y\ar^{1,a,b}@(ur,dr)}$ && \begin{tabular}{c}Total action \\ Not simple  \end{tabular}\\
\hline
$\begin{array}{|c|cc|}\hline
& x&y\\\hline
1&x&y\\a&\emptyset&\emptyset\\ b&\emptyset&\emptyset\\ \hline \end{array}$
&& $\xymatrix{x\ar_{1}@(ul,dl)&y\ar^{1}@(ur,dr)}$ && \begin{tabular}{c}Partial action \\ Not simple  \end{tabular}\\
\hline
\end{tabular}
\end{center}

For $|X|=3$: out of 4096 tables, 10 actions, 3 classes, none simple.

\begin{center}
\begin{tabular}{|c|ccc|c|}
\hline 
$\begin{array}{|c|ccc|}\hline
& x&y&z\\\hline
1&x&y&z\\a&x&y&z\\ b&x&y&z\\ \hline \end{array}$
&& $\xymatrix{x\ar_{1,a,b}@(ul,dl)& & y\ar^{1,a,b}@(ur,dr) \\ & z\ar^{1,a,b}@(ul,ur)}$ && \begin{tabular}{c}Total action \\ Not simple  \end{tabular}\\
\hline
$\begin{array}{|c|ccc|}\hline
& x&y&z\\\hline
1&x&y&z\\a&x&z&z\\ b&z&z&z\\ \hline \end{array}$
&& $\xymatrix{x\ar_{1,a}@(ul,dl)\ar^{b}[dr]& & y\ar^{1}@(ur,dr)\ar^{a}[dl] \\ & z\ar_{1,a,b}@(dl,dr)}$ && \begin{tabular}{c}Total action \\ Not simple  \end{tabular}\\
\hline
$\begin{array}{|c|ccc|}\hline
& x&y&z\\\hline
1&x&y&z\\a&\emptyset&\emptyset&\emptyset\\ b&\emptyset&\emptyset&\emptyset\\ \hline \end{array}$
&& $\xymatrix{x\ar_{1}@(ul,dl)& & y\ar^{1}@(ur,dr) \\ & z\ar^{1}@(ul,ur)}$ && \begin{tabular}{c}Partial action \\ Not simple  \end{tabular}\\
\hline
\end{tabular}
\end{center}

\subsection{Monoid $M_6$: $a^2=a,\ b^2=a,\ ab=ba=b$}

Cayley table:
$$\begin{array}{c|ccc}&1&a&b \\ \hline 1&1&a&b\\ a&a&a&b\\ b&b&b&a \end{array}$$

For $|X|=1$: 2 classes, both simple.

\begin{center}
\begin{tabular}{|c|ccc|c|}
\hline 
$\begin{array}{|c|c|}\hline
& x\\\hline
1&x\\a&x\\ b&x\\ \hline \end{array}$
&& $\xymatrix{x\ar^{1,a,b}@(ur,dr)}$ && \begin{tabular}{c}Total action \\ Simple: YES  \end{tabular}\\
\hline
$\begin{array}{|c|c|}\hline
& x\\\hline
1&x\\a&\emptyset\\ b&\emptyset\\ \hline \end{array}$
&& $\xymatrix{x\ar^{1}@(ur,dr)}$ && \begin{tabular}{c}Partial action \\ Simple: YES  \end{tabular}\\
\hline
\end{tabular}
\end{center}

For $|X|=2$: out of 81 tables, 5 actions, 3 classes, 1 simple.

\begin{center}
\begin{tabular}{|c|ccc|c|}
\hline 
$\begin{array}{|c|cc|}\hline
& x&y\\\hline
1&x&y\\a&x&y\\ b&x&y\\ \hline \end{array}$
&& $\xymatrix{x\ar_{1,a,b}@(ul,dl)&y\ar^{1,a,b}@(ur,dr)}$ && \begin{tabular}{c}Total action \\ Not simple  \end{tabular}\\
\hline
$\begin{array}{|c|cc|}\hline
& x&y\\\hline
1&x&y\\a&x&x\\ b&x&x\\ \hline \end{array}$
&& $\xymatrix{x\ar_{1,a,b}@(ul,dl)&y\ar^{1}@(ur,dr)\ar_{a,b}[l]}$ && \begin{tabular}{c}Total action \\ Not simple  \end{tabular}\\
\hline
$\begin{array}{|c|cc|}\hline
& x&y\\\hline
1&x&y\\a&x&y\\ b&y&x\\ \hline \end{array}$
&& $\xymatrix{x\ar_{1,a}@(ul,dl)\ar@/^/^{b}[r]&y\ar^{1,a}@(ur,dr)\ar@/^/^{b}[l]}$ && \begin{tabular}{c}Total action \\ Simple: YES  \end{tabular}\\
\hline
$\begin{array}{|c|cc|}\hline
& x&y\\\hline
1&x&y\\a&\emptyset&\emptyset\\ b&\emptyset&\emptyset\\ \hline \end{array}$
&& $\xymatrix{x\ar_{1}@(ul,dl)&y\ar^{1}@(ur,dr)}$ && \begin{tabular}{c}Partial action \\ Not simple  \end{tabular}\\
\hline
$\begin{array}{|c|cc|}\hline
& x&y\\\hline
1&x&y\\a&x&y\\ b&\emptyset&\emptyset\\ \hline \end{array}$
&& $\xymatrix{x\ar_{1,a}@(ul,dl)&y\ar^{1,a}@(ur,dr)}$ && \begin{tabular}{c}Partial action \\ Not simple  \end{tabular}\\
\hline
\end{tabular}
\end{center}

For $|X|=3$: out of 4096 tables, 18 actions, 5 classes, none simple.

\begin{center}
\begin{tabular}{|c|ccc|c|}
\hline 
$\begin{array}{|c|ccc|}\hline
& x&y&z\\\hline
1&x&y&z\\a&x&y&z\\ b&x&y&z\\ \hline \end{array}$
&& $\xymatrix{x\ar_{1,a,b}@(ul,dl)& & y\ar^{1,a,b}@(ur,dr) \\ & z\ar^{1,a,b}@(ul,ur)}$ && \begin{tabular}{c}Total action \\ Not simple  \end{tabular}\\
\hline
$\begin{array}{|c|ccc|}\hline
& x&y&z\\\hline
1&x&y&z\\a&x&x&x\\ b&x&x&x\\ \hline \end{array}$
&& $\xymatrix{x\ar_{1,a,b}@(ul,dl)& & y\ar^{1}@(ur,dr)\ar_{a,b}[ll] \\ & z\ar^{1}@(ur,dr)\ar_{a,b}[lu]}$ && \begin{tabular}{c}Total action \\ Not simple  \end{tabular}\\
\hline
$\begin{array}{|c|ccc|}\hline
& x&y&z\\\hline
1&x&y&z\\a&x&y&z\\ b&y&x&y\\ \hline \end{array}$
&& $\xymatrix{x\ar_{1,a}@(ul,dl)\ar@/^/^{b}[rr]& & y\ar^{1,a}@(ur,dr)\ar@/^/^{b}[ll] \\ & z\ar^{1,a,b}@(ul,ur)}$ && \begin{tabular}{c}Total action \\ Not simple  \end{tabular}\\
\hline
$\begin{array}{|c|ccc|}\hline
& x&y&z\\\hline
1&x&y&z\\a&x&y&z\\ b&x&z&x\\ \hline \end{array}$
&& $\xymatrix{x\ar_{1,a,b}@(ul,dl)& & y\ar^{1,a}@(ur,dr)\ar^{b}[ld] \\ & z\ar^{1,a}@(ul,ur)\ar_{b}[ru]}$ && \begin{tabular}{c}Total action \\ Not simple  \end{tabular}\\
\hline
$\begin{array}{|c|ccc|}\hline
& x&y&z\\\hline
1&x&y&z\\a&\emptyset&\emptyset&\emptyset\\ b&\emptyset&\emptyset&\emptyset\\ \hline \end{array}$
&& $\xymatrix{x\ar_{1}@(ul,dl)& & y\ar^{1}@(ur,dr) \\ & z\ar^{1}@(ul,ur)}$ && \begin{tabular}{c}Partial action \\ Not simple  \end{tabular}\\
\hline
\end{tabular}
\end{center}

\subsection{Global summary for all monoids}
We condense all the observed information in the following table.
\begin{center}
\begin{tabular}{|c|c|c|c|c|c|}
\hline
\textbf{Monoid} & $|X|$ & \textbf{Actions} & \textbf{Classes} & \textbf{Indecomposable} & \textbf{Simple} \\
\hline
 {$M_1$} & 1 & 2 & 2 & 2 & 2 \\
& 2 & 7 & 5 & 1 & 1 \\
& 3 & 35 & 11 & 0 & 0 \\
\hline
 {$M_2$} & 1 & 3 & 3 & 3 & 3 \\
& 2 & 23 & 13 & 0 & 0 \\
& 3 & 57 & 14 & 0 & 0 \\
\hline
 {$M_3$} & 1 & 2 & 2 & 2 & 2 \\
& 2 & 5 & 3 & 1 & 1 \\
& 3 & 22 & 5 & 0 & 0 \\
\hline
\end{tabular}
\end{center}

\begin{center}
\begin{tabular}{|c|c|c|c|c|c|}
\hline 
\textbf{Monoid} & $|X|$ & \textbf{Actions} & \textbf{Classes} & \textbf{Indecomposable} & \textbf{Simple} \\
\hline
 {$M_4$} & 1 & 2 & 2 & 2 & 2 \\
& 2 & 4 & 3 & 0 & 0 \\
& 3 & 16 & 5 & 0 & 0 \\
\hline
 {$M_5$} & 1 & 2 & 2 & 2 & 2 \\
& 2 & 3 & 2 & 0 & 0 \\
& 3 & 10 & 3 & 0 & 0 \\
\hline
 {$M_6$} & 1 & 2 & 2 & 2 & 2 \\
& 2 & 5 & 3 & 1 & 1 \\
& 3 & 18 & 5 & 0 & 0 \\
\hline
\end{tabular}
\end{center}

\section{General results}

In this section we present three general results that emerged when performing the calculations with small monoids. These lemmas are valid for any finite monoid $M$ and any $M$-set (total or partial).

\begin{lema}
Let $X$ be a non-empty simple $M$-set (total or partial). Then
\[
|X| \leq |M|.
\]
\end{lema}

\begin{proof}
Fix an element $x_0 \in X$ (if the action is partial, take an element where at least the identity is defined; the identity is always defined). Consider the \emph{orbit} of $x_0$ under the action:
\[
\mathcal{O}(x_0) := \{ m \cdot x_0 \mid m \in M \text{ and } m\cdot x_0 \text{ is defined} \}.
\]
This orbit is invariant under the action of $M$: if $y = m\cdot x_0$ is defined and $n\in M$ is such that $n\cdot y$ is defined, then by the compatibility of the action we have
\[
n\cdot y = n\cdot (m\cdot x_0) = (nm)\cdot x_0 \in \mathcal{O}(x_0).
\]
Since $X$ is simple and $\mathcal{O}(x_0)$ is non-empty (it contains $1\cdot x_0 = x_0$), we necessarily have $\mathcal{O}(x_0) = X$. Therefore, the map
\[
\phi: M \longrightarrow X, \qquad \phi(m) = m\cdot x_0
\]
is surjective (in the sense that every element of $X$ is reached by some $m$ where the action is defined). Hence $|X| \leq |M|$.
\end{proof}

\begin{lema}
Let $\Omega$ be a total $M$-set and let $\Lambda \subseteq \Omega$ be a sub-$M$-set (partial). Then $\Lambda$ is necessarily total.
\end{lema}

\begin{proof}
Let $\lambda \in \Lambda$ and $m\in M$. Since $\Omega$ is total, $\lambda \cdot_\Omega m$ is defined and belongs to $\Omega$. The condition that $\Lambda$ is a sub-$M$-set (i.e., that the inclusion $\iota: \Lambda \hookrightarrow \Omega$ is a morphism of partial $M$-sets) requires that
\[
\iota(\lambda \cdot_\Lambda m) = \iota(\lambda) \cdot_\Omega m = \lambda \cdot_\Omega m.
\]
The right-hand side is defined, so the left-hand side must also be defined. Since $\iota$ is the inclusion, we have $\iota(\lambda \cdot_\Lambda m) = \lambda \cdot_\Lambda m$. Therefore, $\lambda \cdot_\Lambda m$ is defined for all $\lambda \in \Lambda$ and $m\in M$, proving that $\Lambda$ is total.
\end{proof}

\begin{lema}
Let $\Omega$ be a non-empty partial $M$-set. Then $\Omega$ is simple if and only if $\Omega$ is indecomposable.
\end{lema}

\begin{proof}
$(\Rightarrow)$ If $\Omega$ is simple, it has no non-empty proper $M$-subsets, so it cannot be expressed as a disjoint union of two non-empty $M$-subsets. Hence it is indecomposable.

$(\Leftarrow)$ Suppose $\Omega$ is indecomposable. If it were not simple, there would exist a non-empty proper $M$-subset $\Lambda \subset \Omega$. Consider its complement $X \setminus \Lambda$. The action of $M$ restricts to this complement as follows: for $z \in \Omega \setminus \Lambda$ and $m\in M$, define
\[
z \cdot_{(\Omega \setminus \Lambda)} m =
\begin{cases}
z \cdot_\Omega m & \text{if } z \cdot_\Omega m \text{ is defined and belongs to } \Omega \setminus \Lambda,\\
\text{undefined} & \text{otherwise.}
\end{cases}
\]
It is immediate to verify that this defines a partial $M$-set structure on $\Omega \setminus \Lambda$. Moreover, $\Omega$ is the external disjoint union of $\Lambda$ and $\Omega \setminus \Lambda$ as $M$-sets, i.e.,
\[
\Omega \cong \Lambda \sqcup (\Omega \setminus \Lambda),
\]
which contradicts the indecomposability of $\Omega$. Therefore, $\Omega$ is simple.
\end{proof}

\subsection*{Final remark}

An important consequence of Lemma 2 is that \emph{no genuine partial $M$-set (i.e., one that is not total) can be a sub-$M$-set of a total one}. In categorical terms, the inclusion of the category of total actions into that of partial actions preserves simplicity, but has no inverse. The simple objects that appear in the partial world and not in the total one are precisely those we have identified in the previous sections: the partial unitary action on a point, and the exchange action on the monoid $M_1$ on two elements.

The purpose of these notes is to document the computational classification of partial and total actions of the non-group monoids of order 2 and 3, and to point out some structural phenomena that emerge naturally from this analysis. We emphasize that these are observations derived from an exploratory process; although proofs are provided for the stated results, possible improvements in formalization and generalization remain open. Nevertheless, we consider that this work lays the foundations for a solid line of research, both in the study of larger monoids and in the connection with the theory of Burnside rings and $C$-sets.

\section*{Acknowledgements}

\section*{Declarations}

 \begin{bf}Artificial intelligence assistance declaration\end{bf}\\
During the preparation of this manuscript, the authors used DeepSeek-V4-Flash for language editing, grammar checking, and assistance in the formulation and refinement of the exposition. After using this tool, the authors reviewed, verified, and edited the content as necessary, and take full responsibility for the final content of this publication.\\

\indent \begin{bf}\hspace{0.25in}Ethical approval\end{bf}\\
Not applicable.\\

\begin{bf}Conflict of interest\end{bf}\\
The author declares that there are no financial conflicts of interest or known personal relationships that could have influenced the work reported in this article.\\

\begin{bf}Author contributions\end{bf}\\
N/A\\

\begin{bf}Funding\end{bf}\\
This research has been supported by the Centro de Ciencias Matemáticas, UNAM Campus Morelia, which has provided the necessary academic environment for the author to carry out his postdoctoral stay. Additionally, the research has been funded by SECIHTI, Postdoctoral Fellowship No. I1200/111/2024, whose support has been fundamental to carry out these research activities. \\

\begin{bf}Data availability\end{bf}\\
Not applicable.

 \bibliographystyle{plain}
\bibliography{bibliografia}

\end{document}